\documentclass[11pt,a4paper,reqno]{article}

\usepackage{amsmath,amsthm,amsbsy}
\usepackage{enumerate}

\usepackage[bitstream-charter]{mathdesign}

\usepackage[colorlinks,linkcolor=red,anchorcolor=blue,citecolor=cyan]{hyperref}
\usepackage[numbers,sort&compress]{natbib}

\theoremstyle{plain}\newtheorem{definition}{Definition}[section]
\theoremstyle{definition}\newtheorem{theorem}{Theorem}[section]
\theoremstyle{plain}\newtheorem{lemma}[theorem]{Lemma}
\theoremstyle{plain}
\theoremstyle{plain}\newtheorem{prop}[theorem]{Proposition}
\theoremstyle{remark}\newtheorem{remark}{Remark}[section]
\usepackage{xcolor}

\newcommand{\R}{\mathbb{R}}

\newcommand{\ti}{\nabla}
\newcommand{\ep}{\varepsilon}

\newcommand{\qand}{\qquad \text{and} \qquad}

\newcommand{\set}[1]{\left\{#1\right\}}
\newcommand{\norm}[1]{{\left\Vert #1 \right\Vert}}

\DeclareMathOperator*{\Div}{div}

\numberwithin{equation}{section}
\begin{document}

\title{Improved bounds for box dimensions of potential singular points to the Navier--Stokes equations}

\author{Yanqing Wang\footnote{Department of Mathematics and Information Science, Zhengzhou University of Light Industry, Zhengzhou, Henan  450002,  P. R. China. Email: wangyanqing20056@gmail.com}  ~~ \& ~ Minsuk Yang\footnote{Correspondence author, Department of Mathematics, Yonsei University, 50 Yonsei-ro Seodaemun-gu, Seoul, Republic of Korea. Email: m.yang@yonsei.ac.kr }}

\date{}

\maketitle

\begin{abstract}
In this paper, we study the potential singular points of interior and boundary suitable weak solutions to the 3D Navier--Stokes equations.
It is shown that upper box dimension of interior singular points and boundary singular points are bounded by $7/6$ and $10/9$, respectively.
Both proofs rely on recent progress of $\varepsilon$-regularity criteria at one scale.
\end{abstract}


\section{Introduction}
\label{S1}
\setcounter{section}{1}\setcounter{equation}{0}

We  consider  the following three-dimensional incompressible Navier--Stokes equations
\begin{equation}
\label{NS}
\left\{
\begin{split}
&u_{t} -\Delta  u+ u\cdot\ti u+\nabla \Pi=0, \\
&\Div u=0,\\
&u|_{t=0}=u_0,
\end{split}
\right.
\end{equation}
where $u$ is the velocity field, $\Pi$ is the scalar pressure, and the initial velocity $u_0$ satisfies the condition $\Div u_0=0$.
We investigate the regularity problem of the Navier--Stokes equations.
Our main objective of this paper is to further lower the box dimension of potential singular points of suitable weak solutions to the Navier--Stokes equations.

A point $z=(x,t)$ is said to be regular if $u$ is H\"older continuous at a neighborhood of $z$.
Otherwise, it is called singular.
In a celebrated work \cite{[CKN]},  Caffarelli, Kohn and Nirenberg obtained two $\ep$-regularity criteria to  suitable weak solutions of \eqref{NS}: $z=0$ is regular point provided that one of the following conditions holds for an absolute positive constant $\ep$,
\begin{equation}
\label{CKN}		
\norm{u}_{L^{3}(Q(1))}+\|u\Pi\|_{L^1(Q(1))}+\|\Pi\|_{L^{1,5/4}(Q(1))} < \ep.
\end{equation}
and
\begin{equation}
\label{ckn2}
\limsup_{r\rightarrow0}  \norm{\nabla u}_{L^{2}(Q(r))} < \ep.
\end{equation}
Lin \cite{[Lin]}, Ladyzhenskaya and Seregin \cite{[LS]} gave an alternative condition
\begin{equation}
\label{Lin}
\norm{u}_{L^{3}(Q(1))}+\|\Pi\|_{L^{3/2}(Q(1))} < \ep.
\end{equation}
instead of \eqref{CKN}.
Recently, Guevara and Phuc \cite{[GP]} improved \eqref{CKN} and \eqref{Lin} to
\begin{equation}
\label{GP}
\norm{u}_{L^{2p, 2q} (Q(1))}+\|\Pi\|_{L^{p, q}(Q(1))} < \ep,
\end{equation}
for some $p, q$ satisfying ${3}/{p}+{2}/{q} = 7/2$ and $1\leq q\leq2$.
Subsequently, authors in \cite{[HWZ]} found  that \eqref{GP} can be replaced by
\begin{equation}
\label{optical}
\norm{u}_{L^{p,q}(Q(1))}+\|\Pi\|_{L^{1}(Q(1))}<\ep
\end{equation}
for some $p, q$ satisfying $1\leq 2/p+3/q <2$ and $1\leq p,\,q\leq\infty$.
Other interior regularity criteria similar to \eqref{CKN} can be found in \cite{[Vasseur],[CV],[WZ]}.
Since the gradient of the pressure appears in \eqref{NS}, one can replace $\Pi$ as $ \Pi-\overline{\Pi}_{B(1)}$ by subtracting an average in \eqref{CKN}, \eqref{Lin}, \eqref{GP}, and \eqref{optical}.
Let $\mathcal{S}_{i}$ denote the possible interior singular points of suitable weak solutions to the 3D Navier--Stokes equations.
One can use the condition \eqref{CKN} and \eqref{ckn2} to obtain that
\[
\dim_{H}(\mathcal{S}_{i})\leq 1 \qand \dim_{B}(\mathcal{S}_{i})\leq5/3,
\]
where $\dim_{H}(S)$ and $\dim_{B}(S)$ denote the Hausdorff dimension and box dimension of a set $S$, respectively.
For the background of fractal dimension, we refer the reader to \cite{[Falconer]}.
The relationship between Hausdorff dimension and box dimension is that the former  is less than the latter.

In the past decade, starting from Kukavica's wrok \cite{[Kukavica]}, several authors try to lower the box dimension of potential interior singular points for suitable weak solutions to the 3D Navier--Stokes equations to $1$ (Kukavica  \cite{[Kukavica]} $(135/82\approx 1.65)$; Kukavica and  Pei \cite{[KP]} $45/29(\approx1.55)$; Koh and Yang \cite{[KY16]} $95/63( 1.51)$; Wang and Wu $(360/277\approx1.30)$; Ren, Wang and Wu \cite{[RWW]} $(975/758\approx1.29)$; He, Wang and Zhou  \cite{[HWZ]} $(2400/1903\approx1.261)$).

In this paper, we also consider the potential boundary singular points of boundary suitable weak solutions to the Navier--Stokes equations.
In particular,   the Navier--Stokes equations with no-slip conditions  are given by
\begin{equation}
\label{NSB}
\left\{
\begin{split}
&u_{t} -\Delta  u+ u\cdot\ti
u +\nabla \Pi=0, \mathbb{R}^{3}_{+}\times(0,T)\\
&\Div u=0,~~u|_{\partial\mathbb{R}^{3}_{+}}=0\\
&u|_{t=0}=u_0.
\end{split}
\right.
\end{equation}

For the boundary regularity problem of the Navier--Stokes equations \eqref{NSB}, Seregin \cite{[Seregin]} obtained the following two $\ep$-regularity conditions
\begin{align}
\label{sb1}
\limsup_{\mu\rightarrow0}  \norm{\nabla u}_{L^{2}(Q^{+}(\mu))} &< \ep, \\
\label{sb2}
\norm{u}_{L^{3}(Q^{+}(1))}+\|\Pi\|_{L^{3/2}(Q^{+}(1))} &< \ep.
\end{align}
After that, Gustafson, Kang and Tsai \cite{[GKT]} obtained the $\ep$-regularity condition
\begin{equation}
\label{gktb}
\norm{u}_{L^{3}(Q^{+}(1))}+\|\Pi\|_{L^{\lambda,\kappa^{\ast}}(Q^{+}(1))} < \ep,
\end{equation}
where $\lambda$ and $\kappa^{\ast}$ are defined in \eqref{gktc1}.
Let $\mathcal{S}_{b}$ denote the possible boundary singular points of boundary suitable weak solutions to the 3D Navier--Stokes equations.
Using these $\ep$-regularity conditions, Choe and Yang \cite{[CY18]} recently obtained a result about upper box dimension $\dim_{B}(\mathcal{S}_{b})$.
In this paper we improve all the previous bounds for the interior and boundary singular points.

The rest of this paper is organized as follows.
In Section 2, we give the precise statement of our main results.
In Section 3, we introduce relevant notation and definitions of suitable weak solutions to \eqref{NS} and \eqref{NSB}, respectively.
In Section 4, we present a few auxiliary lemmas which are interpolation inequalities and decay estimates for the pressure.
In Section 5, we prove Proposition \ref{boxprop} and deduce Theorem \ref{the1.1}, which is an estimation of box dimension of $\mathcal{S}_{i}$ in \eqref{NS}.
In Section 6, we prove Theorem \ref{brc} and Proposition \ref{boxpropb} and deduce Theorem \ref{the1.2}, which is an estimation of box dimension of $\mathcal{S}_{b}$ in \eqref{NSB}.

\section{Main results}
\label{S2}

In this section we state our main results of this paper.

\begin{theorem}
\label{the1.1}
The upper box dimension of  $\mathcal{S}_{i}$ in \eqref{NS} is at most $7/6$.
\end{theorem}

\begin{remark}
The bound $7/6$ is better than the previous results obtained in \cite{[KP],[RS2],[Kukavica],[HWZ],[RWW],[KY16],[WW2]}.
It is a direct consequence of the following regularity criterion.
\end{remark}

\begin{prop}
\label{boxprop}
Suppose that $(u, \,\Pi)$ is a suitable weak solution to (\ref{NS}).
Then for each $\gamma <1/2$ there exist positive numbers $\ep_{1}$ and $r_0<1$ such that $z=(x,t)$ is a regular point if for some $r<r_0$
\begin{equation}
\label{condi}
\int_{Q(z,r)} |\nabla u|^{2} + |u|^{ 10/3} + |\Pi-\overline{\Pi}_{ B (x,r)}|^{ 5/3} + |\nabla \Pi|^{5/4} dxds
< r^{5/3-\gamma}\ep_{1}
\end{equation}
where $Q(z,r)$ denotes a parabolic cylinder (see the next section for notations and definitions).
\end{prop}	

The proof of Proposition \ref{boxprop} is different from recent arguments in \cite{[KY16],[WW2],[RWW],[HWZ]}.
The key point is to apply the following $\ep$-regularity criteria, for any $\delta>0$,
\[
\norm{u}_{L^{2p}(Q(1))} + \|\Pi\|_{L^{p}(Q(1))} < \ep \qand p > \frac{5}{2} + \delta,
\]
to establish an iteration scheme.
Our starting point is the following $\ep$-regularity criterion
\begin{equation}
\label{special}
\norm{\nabla u}_{L^{2}(Q(1))} + \norm{u}_{L^{2}(Q(1))} + \|\nabla\Pi\|_{L^{1,3/2}(Q(1))} < \ep,
\end{equation}
which is  derived from
\begin{equation}
\label{specia2}
\norm{u}_{L^{2,6}(Q(1))} + \|\nabla\Pi\|_{L^{1,3/2}(Q(1))} < \ep.
\end{equation}
It is worth noting that we get \eqref{specia2} by \eqref{optical} and the Poinca\'e--Sobolev inequality.

We briefly illustrate our strategy in the proof of Proposition \ref{boxprop}, see Section \ref{S5} for its detailed proof.
We bound
\[
{(\theta\rho)}^{-3/2} \norm{u}_{L^{2}(Q(\theta\rho))}
\]
via interpolation inequality  \eqref{cp} and hypothesis \eqref{condi}.
And then making use of the divergence free condition, we establish a pressure decay estimate \eqref{presure} in terms of $\nabla \Pi$.
Using the same scaled quantities, we get smallness of ${(\theta\rho)}^{-3/2}\norm{u}_{L^{2}(Q(\theta\rho))}$ as well as
\[
(\theta\rho)^{-1}(\|\nabla\Pi\|_{L^{5/4}(Q(\theta\rho))}+\norm{\nabla u}_{L^{2}(Q(\theta\rho))}).
\]
This together with \eqref{special} completes the proof of Proposition \ref{boxprop}.

Before we turn out attention to the boundary case, we give one more remark.

\begin{remark}
As was observed in \cite[Remark 1.4, p1762]{[WW2]}, it seems that it is useful to use the quantity $\norm{\nabla u}_{L^{2}}$ instead of $\|u\|_{L^{\infty,2}}$ in estimating box dimension of the singular set.
The advantage of $\norm{u}_{L^{2}(Q(1))}$ and $\|\nabla\Pi\|_{L^{1,3/2}(Q(1))}$ in \eqref{special} is the absence of $\norm{u}_{L^{\infty,2}}$ in first part of  estimate \eqref{cp} and \eqref{presure}.
This helps us to make full  use of  $\norm{\nabla u}_{L^{2}}$ in our proof.
One can show that $\dim_B(\mathcal{S}_{i}) \le 37/30$ via a combination of the proof described above and the $\ep$-regularity criterion
\begin{equation}
\label{special4}
\norm{u}_{L^{5/2+\delta}(Q(1))} + \|\nabla\Pi\|_{L^{5/4}(Q(1))} < \ep.
\end{equation}
\end{remark}
	
For the boundary case \eqref{NSB} we have the following theorem.

\begin{theorem}
\label{the1.2}
The upper box dimension of  $\mathcal{S}_{b}$ in \eqref{NSB} is at most $10/9.$
\end{theorem}

\begin{remark}
The bound $10/9$ is better than the previous result obtained in \cite{[CY18]}.
It is a direct consequence of the following regularity criterion.
\end{remark}

\begin{prop}
\label{boxpropb}
Suppose that $(u, \,\Pi)$ is a suitable weak solution to (\ref{NSB}).
Then there exist positive numbers $\ep_2$ and $r_0<1$ such that $z$ is a regular point if for some $0<r<r_0$
\begin{equation}
\label{condb}
\int_{ Q^{+} (x,t,r)} |\nabla u |^{2} +| u |^{ 10/3}+|\Pi-\overline{\Pi}_{ B^{+} (x,r)} |^{ 5/3} + |\nabla \Pi| ^{5/4}dxds
< r^{10/9}\ep_{2}.
\end{equation}
\end{prop}

Compared with the proof of Proposition \ref{boxprop}, the proof of Proposition \ref{boxpropb} includes a new ingredient of an application of new $\ep$-regularity  criteria below.

\begin{theorem}
\label{brct}
Suppose that $(u, \Pi)$ is a suitable weak solution to the 3D Navier--Stokes equations \eqref{NSB} in $Q^{+}(1)$.
Then there exists an absolute positive constant $\ep$ such that $z=0$ is a regular point if
\begin{equation}
\label{opticalbt}
\norm{\nabla u}_{L^{p,q}(Q^{+}(1))}+
\|\Pi-\overline{\Pi}_{B^+_{x,r}}(s)\|_{L^{\lambda,\kappa^{\ast}}(Q^{+}(1))} <\ep
\end{equation}
for some $p, q$ satisfying $2\leq 2/q+3/p <3$ and $1\leq p,\,q\leq\infty$.
\end{theorem}

\begin{remark}
A special case is
\[
\norm{\nabla u}_{L^{2}(Q^{+}(1))}+\|\nabla\Pi \|_{L^{5/4,5/4}(Q^{+}(1))} <\ep,
\]
which can be used to show that $\dim_{B}(\mathcal{S}_{b})\leq5/4$.
\end{remark}

The Poincar\'e--Sobolev inequality guarantees Theorem \ref{brct} follows from the next theorem, which is of independent interest.

\begin{theorem}
\label{brc}
Suppose that $(u, \Pi)$ is a suitable weak solution to the 3D Navier--Stokes equations \eqref{NSB} in $Q^{+}(1)$.
Then there exists an absolute positive constant $\ep$ such that $z=0$ is a regular point if
\begin{equation}
\label{opticalb}
\norm{u}_{L^{p,q}(Q^{+}(1))}+
\|\Pi-\overline{\Pi}_{B^+_{x,r}}(s)\|_{L^{\lambda,\kappa^{\ast}}(Q^{+}(1))} <\ep
\end{equation}
for some $p, q$ satisfying $1\leq 2/q+3/p <2$ and $1\leq p,\,q\leq\infty$.
\end{theorem}

\begin{remark}
Theorem \ref{brc} is an improvement of \eqref{sb2} and \eqref{gktb}.
The proof is in part inspired by recent results \eqref{GP} and \eqref{optical}.
\end{remark}

\section{Notations and definitions}
\label{S3}

We denote by $L^{q}(-T,0; X)$, $1 \le q \le \infty$, the set of measurable functions on the interval $(-T,0)$ with values in $X$ and $\|f(t,\cdot)\|_{X} \in L^{q}(-T,0)$.
We denote by $z = (x,t) \in \mathbb{R}^{3} \times (-T,0)$ a space-time point and denote balls and cylinders by
\begin{align*}
B(x,r) &= \{y\in \mathbb{R}^{3} : |x-y|\leq r\}, \\
Q(z,r) &= B(x,r)\times(t-r^{2}, t).
\end{align*}

We recall the definition of suitable weak solutions to \eqref{NS}.

\begin{definition}
\label{defi}
A  pair   $(u, \Pi)$  is called a suitable weak solution to the Navier--Stokes equations \eqref{NS} provided the following conditions are satisfied:
\begin{enumerate}[(1)]
\item
\label{SWS1}
$u \in L^{\infty}(-T,\,0;\,L^{2}(\mathbb{R}^{3})) \cap L^{2}(-T,\,0;\,\dot{H}^{1}(\mathbb{R}^{3}))$, $\Pi \in
L^{3/2}(-T,\,0;L^{3/2}(\mathbb{R}^{3}))$.
\item
\label{SWS2}
$(u, \Pi)$~solves (\ref{NS}) in $\mathbb{R}^{3}\times (-T,\,0) $ in the sense of distributions.
\item
\label{SWS3}
$(u, \Pi)$ satisfies the following inequality, for a.e. $t\in(-T,0)$,
\begin{equation}
\label{loc}
\begin{split}
&\int_{\mathbb{R}^{3}} |u(x,t)|^{2} \phi(x,t) dx
+2\int^{t}_{-T}\int_{\mathbb{R} ^{3 }} |\nabla u|^{2}\phi  dxds \\
&\leq \int^{t}_{-T }\int_{\mathbb{R}^{3}} |u|^{2} (\partial_{s}\phi+\Delta \phi)dxds
+ \int^{t}_{-T } \int_{\mathbb{R}^{3}}u\cdot\nabla\phi (|u|^{2} +2\Pi)dxds,
\end{split}
\end{equation}
where non-negative function $\phi(x,s)\in C_{0}^{\infty}(\mathbb{R}^{3}\times (-T,0) )$.
\end{enumerate}
\end{definition}

In the light of the natural scaling property of the Navier--Stokes  equations, we introduce the following scaling invariant quantities;
\begin{equation}
\label{scaling}
\begin{split}
A(r) &= \sup_{t-r^{2}\leq s<t} r^{-1} \int_{B(x,r)}|u(y,s)|^{2}dy, \\
E(r) &= r^{-1}  \int_{Q(z,r)}|\nabla u(y,s)|^{2}dyds, \\
E_{p}(r) &= r^{p-5} \int_{Q(z,r)}|u(y,s)|^{p}dyds, \\
P_{5/4}(\nabla\Pi,r) &= r^{-1} \left(\int_{Q(z,r)}|\nabla\Pi|^{5/4}dyds\right)^{4/5}, \\
P_{1,5/4}(\nabla\Pi,r) &= r^{-1} \int^{t}_{t-r^{2}}\left(\int_{B(x,r)} |\nabla\Pi|^{5/4}dy\right)^{4/5}ds.
\end{split}
\end{equation}

For any boundary point $z=(x,t)=(x_{1},x_{2},0,t) \in \partial\R^3_+\times (0,T)$, we denote half balls and half cylinders by
\begin{align*}
B^{+}(x,r) &=\{y \in B(x,r): y_3>0 \}, \\
Q^{+}(z,r) &=\{(y,t) \in Q(z,r): y_3>0 \}.
\end{align*}

We recall the definition of boundary suitable weak solutions to \eqref{NSB}.

\begin{definition}
\label{sws-3dnse}
Let $\Omega=\R^3_+$ and $Q_T=\R^3_+\times [-T,0)$. A pair of
$(u,\Pi)$ is a suitable weak solution to \eqref{NSB} if the
following conditions are satisfied:
\begin{enumerate}[(1)]
\item  The functions $u: Q_T\rightarrow \mathbb{R}^3$ and $\Pi : Q_T \rightarrow \mathbb{R}$ satisfy
\begin{align*}
&u \in L^{\infty}\left(I;L^{2}(\Omega)\right)\cap L^{2}\left(I;W^{1,2}(\Omega)\right), \quad
\Pi \in L^{\lambda}\left(I;L^{\kappa^*}(\Omega)\right), \\
&\nabla^{2}u \in L^{\lambda}\left(I;L^{k}(\Omega)\right),\quad
\nabla \Pi \in L^{\lambda}\left(I;L^{\kappa}(\Omega)\right),
\end{align*}
where $\lambda$, $\kappa$ and $\kappa^*$ are fixed numbers satisfying
\begin{equation}
\label{gktc1}
1<\lambda<2,\qquad
\frac{2}{\lambda}+\frac{3}{\kappa}=4,\qquad
\frac{1}{\kappa^*}=\frac{1}{\kappa}-\frac{1}{3}.
\end{equation}
\item
$(u, \Pi)$ solves the Navier--Stokes equations in $Q_T$ in the sense of distributions and $u$   satisfy the boundary conditions  in the sense of traces.
\item
$u$ and $\Pi$ satisfy the local energy inequality
\begin{align*}
&\int_{B^{+}{(x,r)}}|u(x,t)|^2 \phi(x,t) dx
+2\int_{t_0}^t\int_{B^{+}{(x,r)}} |\nabla u(x,s)|^{2} \phi(x,s) dx ds \\
&\leq \int_{t_0}^t\int _{B^{+}{(x,r)}} |u|^{2} (\partial_s\phi+\Delta \phi)dxds+ \int_{t_0}^t
\int_{B^{+}{(x,r)}}(|u|^2 + 2\Pi) u\cdot\nabla \phi dxds
\end{align*}
for all $t\in I=(0,T)$ and for all nonnegative function $\phi \in
C_0^{\infty}(\R^3\times R)$.
\end{enumerate}
\end{definition}

We shall use the same scaling invariant quantities \eqref{scaling} for the boundary case replacing $B(x,r)$ and $Q(z,r)$ by $B^{+}(x,r)$ and  $Q^{+}(z,r)$ in \eqref{scaling}.
Readers can tell them in the context.
In addition, we need the following scaling invariant quantities involving the pressure
\begin{align}
\label{tild:p10}
\tilde{D}(r) &= r^{-1} \left(\int_{t-r^2}^{t}
\left(\int_{B^+(x,r)}|\Pi(y,s)-\overline{\Pi}_{B^+(x,r)}(s)|^{\kappa^*}
d y\right)^{\frac{\lambda}{\kappa^*}} ds\right)^{\frac{1}{\lambda}}, \\
\label{tild:p20}
\tilde{D}_1(r) &= r^{-1} \left(\int_{t-r^2}^{t}
\left(\int_{B^+(x,r)}|\nabla \Pi(y,s)|^{\kappa}d y\right)^{\frac{\lambda}{\kappa}} ds\right)^{\frac{1}{\lambda}},
\end{align}
where we have used the notation $\overline{f}_{E}$ which is the average of $f$ over the set $E$.

For $\kappa^*$ and $\lambda$ in \eqref{gktc1}, we denote their
H\"older conjugate $i$ and $j$ through the relations
\[
\frac{1}{i}=1-\frac{1}{\lambda}, \qquad
\frac{1}{j}=1-\frac{1}{\kappa^*}.
\]
Hence, it follows from \eqref{gktc1} that
\begin{equation}
\label{reg:20}
\frac{2}{i}+\frac{3}{j}=2,\qquad 2<i<\infty.
\end{equation}

We end this section by giving a few more notations.
When the center of a ball or a cylinder is located at the origin, we put them simply as
\begin{align*}
B(r) = B(0,r) &\qand Q(r) = Q(0,r) \\
B^{+}(r) = B^{+}(0,r) &\qand Q^{+}(r) = Q^{+}(0,r).
\end{align*}
For simplicity, we also write $\overline{f}_{r} = \overline{f}_{B(r)}$ and
\begin{align*}
\norm{f}_{L^{p,q}(Q(r))} &=\norm{f}_{L^{p}(-r^{2},\,0;\,L^{q}(B(r)))}, \\
\norm{f}_{L^{p}(Q(r))} &= \norm{f}_{L^{p,p}(Q(r))}.
\end{align*}
We shall use the summation convention on repeated indices.
We shall use the notation $A \lesssim B$ if there is a generic positive constant $C$ such that $|A| \le C|B|$.
The generic positive constants $C$ may be different from line to line unless otherwise stated.

\section{Auxiliary lemmas}
\label{S4}

In this section, we present interpolation inequalities and decay estimates for the pressure.

\begin{lemma}[\cite{[HWZ],[RWW]}]
\label{lemma2.1i}
For $0< r \leq \rho/2$ and $2\leq p\leq10/3$,
\begin{equation}
\label{cp}
E_{p}(r) \lesssim \left(\frac{\rho}{r}\right)^{p/2} A^{(p-2)/2}(\rho) E(\rho)
+ \left(\frac{r}{\rho}\right)^{p} A^{p/2}(\rho)
\end{equation}
for the interior case and
\begin{equation}
\label{C3}
E_{p}(r) \lesssim \left(\frac{\rho}{r}\right)^{p/2} A^{(p-2)/2}(\rho) E(\rho)
\end{equation}
for the boundary case, where the implied positive constants does not depend on $r$ and $\rho$.
\end{lemma}

We derive decay estimates in terms of $\nabla\Pi$.
See \cite{[KP]} for a different version.

\begin{lemma}
\label{presurel}
For $0<r\leq \rho/8$,
\begin{equation}
\label{presure}
P_{1,3/2}(\nabla\Pi,r)
\lesssim \left(\dfrac{\rho}{r}\right) E(\rho) + \left(\frac{r}{\rho}\right) P_{5/4}(\nabla\Pi,\rho)
\end{equation}
for the interior case, where the implied positive constant does not depend on $r$ and $\rho$.
\end{lemma}

\begin{proof}
Without loss of generality, we assume that $z=0$.
We consider the usual cut-off function $\phi\in C^{\infty}_{0}(B(\rho/2))$ such that $\phi\equiv1$ on $B(3\rho/8)$ with $0\leq\phi\leq1$ and
\[
|\nabla\phi|^2 + |\nabla^{2}\phi| \leq C\rho^{-2}.
\]
Due to the incompressible condition, we infer that
\[
\partial_{i}\partial_{i}(\partial_{k} \Pi\phi)=-\phi \partial_{i}\partial_{j}\big[\partial_{k}U_{i,j}\big]
+2\partial_{i}\phi\partial_{i}\partial_{k} \Pi+\partial_{k} \Pi\partial_{i}\partial_{i}\phi
\]
where $U_{i,j}=(u_{j}- \overline{{u_{j}}} _{\rho/2})(u_{i}-\overline{{u_{i}}}_{\rho/2})$.
Let $K$ denote the normalized fundamental solution of Laplace equation.
Then for $x\in B(3\rho/8)$
\begin{equation}
\label{pp}
\begin{split}
\partial_{k} \Pi(x)
&= K \ast \{-\phi \partial_{i}\partial_{j}[\partial_{k}U_{i,j}]
+ 2\partial_{i}\phi\partial_{i}\partial_{k} \Pi+\partial_{k} \Pi\partial_{i}\partial_{i}\phi\} \\
&=-\partial_{i}\partial_{j}K \ast (\phi [ \partial_{k}U_{i,j}]) \\
&\quad + 2\partial_{i}K \ast(\partial_{j}\phi [\partial_{k}U_{i,j} ])-K \ast (\partial_{i}\partial_{j}\phi[\partial_{k}U_{i,j}]) \\
&\quad + 2\partial_{k}\partial_{i}K \ast(\partial_{i}\phi\Pi)+2\partial_{i}K \ast(\partial_{k}\partial_{i}\phi\Pi) +\partial_{k}K \ast(\partial_{i}\partial_{i}\phi\Pi)+K \ast(\partial_{i}\partial_{i}\partial_{k}\phi\Pi)\\
&=: \partial_{k} P_{1}(x)+\partial_{k}  P_{2}(x)+\partial_{k}  P_{3}(x).
\end{split}
\end{equation}
Since $\phi(x)=1$, where $x\in B(r)$  ($0<r\leq \rho/4$), we know that there is no singularity in $\partial_{k}  P_{2}(x)$ and $\partial_{k}  P_{3}(x)$.
As a consequence, we have
\begin{align*}
|\partial_{k}  P_{2}(x)| &\leq \rho^{-3} \int_{B(\rho/2)}|\partial_{k}U_{i,j}|dx,\\
|\partial_{k}  P_{3}(x)| &\leq \rho^{-4} \int_{B(\rho/2)}|\Pi|dx,
\end{align*}
and by H\"older's inequality
\begin{align}
\|\partial_{k}  P_{2}(x) \|_{L^{3/2}(B(r))} &\lesssim \left(\frac{r}{\rho}\right)^{2}\| \partial_{k}U_{i,j} \|_{L^{3/2}
(B(\rho/4))},\label{ptip2}\\
\|\partial_{k}  P_{3} (x) \|_{L^{3/2}(B(r))} &\lesssim \left(\frac{r}{\rho}\right)^{2}\| \Pi\|_{L^{5/4,15/7}
(B(\rho/4))}.\label{ptip3}
\end{align}
Using the H\"older inequality and the Poinca\'e-Sobolev inequality, we see that
\begin{equation}
\label{ti2}
\|\partial_{k}U_{i,j}\|_{L^{1,3/2}(Q(\rho/2))}\leq \|u-\overline{u}_{\rho/2}\|_{L^{2,6}(Q(\rho/2))}\norm{\nabla u}_{L^{2}(Q(\rho/2))}\lesssim \|\nabla u\|^{2}_{L^{2}(Q(\rho/2))}.
\end{equation}
It follows from \eqref{ptip2} and \eqref{ti2} that
\begin{equation}
\label{E41}
\|\partial_{k}  P_{2}(x)  \|_{L^{1,3/2}(Q(r))} \lesssim \left(\frac{r}{\rho}\right)^{2} \|\nabla u\|^{2}_{L^{2}(Q(\rho/2))}.
\end{equation}
Notice that $\Pi-\Pi_{\frac{\rho}{2}}$ also satisfies \eqref{pp}, we derive from \eqref{ptip3} that
\begin{equation}
\label{E42}
\begin{split}
\|\partial_{k}  P_{3}(x)  \|_{L^{1,3/2}(Q(r))}
&\lesssim \left(\frac{r}{\rho}\right)^{2}\|  \Pi-\overline{\Pi}_{\frac{\rho}{2}}\|_{L^{5/4,15/7}(Q(\rho/2))} \\
&\lesssim \left(\frac{r}{\rho}\right)^{2}\| \nabla \Pi \|_{L^{5/4}(Q(\rho/2))}.
\end{split}
\end{equation}
According to the Calder\'on--Zygmund theorem and \eqref{ti2}, we get
\begin{equation}
\label{E43}
\|\partial_{k}P_{1}(x)\|_{L^{1,3/2}(Q(r))}
\lesssim \|\partial_{k}U_{i,j}\|_{L^{1,3/2}(Q(\rho/2))}
\lesssim \norm{\nabla u}_{L^{2}(Q(\rho))}^{2}.
\end{equation}
Combining the estimates \eqref{E41}, \eqref{E42}, \eqref{E43}, we get
\begin{equation}
\label{lem2.3}
\begin{split}
\|\partial_{k}\Pi\|_{L^{1,3/2}(Q(r))}
&\lesssim \|\partial_{k}P_{1}(x)\|_{L^{1,3/2}(Q(r))}
+ \|\partial_{k}P_{2}(x)\|_{L^{3/2}(Q(r))}
+ \|\partial_{k}P_{3}(x)\|_{L^{1,3/2}(Q(r))} \\
&\lesssim \norm{\nabla u}_{L^{2}(Q(\rho))}^{2}
+ \left(\frac{r}{\rho}\right)^{2}\|\partial_{k}\Pi\|_{L^{5/4}(Q(\rho))}.
\end{split}
\end{equation}
Hence
\[
P_{1,3/2}(\nabla\Pi,r) \lesssim \left(\dfrac{\rho}{r}\right) E (\rho) + \left(\frac{r}{\rho}\right) P_{5/4}(\nabla\Pi,\rho).
\]
\end{proof}

\begin{lemma}
\label{presurelk}
For $0<r\leq \rho/8$,
\begin{equation}
\label{presurek}
\tilde{D}_1(r)
\lesssim \left(\frac{\rho}{r}\right) E^{1/\lambda}(\rho) A^{1-1/\lambda}(\rho)
+ \left(\frac{r}{\rho}\right) \tilde{D}_1(r),
\end{equation}
where $\lambda$ is defined in \eqref{gktc1} and the implied positive constant does not depend on $r$ and $\rho$.
In addition, this lemma remains valid for $\kappa=1$.
\end{lemma}

\begin{proof}
We get the result by replacing \eqref{pp} with
\begin{equation}
\label{ppk}
\begin{split}
\partial_{k} \Pi(x)
&= K \ast \{-\phi \partial_{i}\partial_{j}[\partial_{k}U_{i,j} ]
+2\partial_{i}\phi\partial_{i}\partial_{k} \Pi+\partial_{k} \Pi\partial_{i}\partial_{i}\phi\}\\
&= -\partial_{i}\partial_{j}K \ast (\phi [ \partial_{k}U_{i,j}] )\\
&\quad +2\partial_{i}K \ast(\partial_{j}\phi [\partial_{k}U_{i,j} ])-K \ast
(\partial_{i}\partial_{j}\phi[\partial_{k}U_{i,j}])\\
&\quad -2\partial_{i}K \ast(\partial_{i}\phi \partial_{k}\Pi) -K \ast(\partial_{i}\partial_{i}\phi \partial_{k}\Pi)\\
&=: \partial_{k} P_{1}(x)+\partial_{k}  P_{2}(x)+\partial_{k}  P_{3}(x).
\end{split}
\end{equation}
and by making a slight variant of the proof of Lemma \ref{presurel}.
\end{proof}

\begin{remark}
This lemma gives an improvement of \cite[Lemma 17, p617]{[GKT]}.
Since the standard parabolic regularization theory are not applicable to the case $L^{1}_{t}$, a particularly interesting case in this lemma is $\kappa=1$,
\end{remark}

In the spirit of above Lemma
\ref{presurel} and \cite[Lemma 4, p11]{[CY18]}, we obtain the following result for the boundary case.
\begin{lemma}\label{preest:100}
If $z=(x,t), x\in\partial\R^3_+, t-\rho^2>0$, and $t<T$, then for $0 \leq r \leq \rho/4$, $\lambda\leq5/4$, and $\kappa\leq15/7$,
\begin{equation}
\label{pbd}
\tilde{D}_1(r) \lesssim \left(\frac{\rho}{r}\right) E^{1/\lambda}(\rho) A^{1-1/\lambda}(\rho)
+ \left(\frac{r}{\rho}\right) \left(E^{1/2}(\rho) + P_{5/4}(\nabla\Pi,\rho)\right),
\end{equation}
where $\lambda$ is defined in \eqref{gktc1} and the implied positive constant does not depend on $r$ and $\rho$.
\end{lemma}

\begin{proof}
A slight variant of the proof \cite[Lemma 11, p608]{[GKT]} provides the estimates
\[
\tilde{D}_1(r)
\lesssim \left(\frac{\rho}{r}\right)E^{1/\lambda}(\rho) A^{1-1/\lambda}(\rho)
+ \left(\frac{r}{\rho}\right) \left(E^{1/2}(\rho)+\rho^{-2} \|\Pi-\overline{\Pi}_{\rho}\|_{L^{\lambda,\kappa}(Q^{+}(\rho))}\right).
\]
From $\lambda\leq5/4$ and $\kappa\leq15/7$   and H\"older's inequality, we see that
\[
\rho^{-2}\|\Pi-\overline{\Pi}_{\rho}\|_{L^{\lambda,\kappa}(Q^{+}(\rho))}
\lesssim \rho^{-1}\|\Pi-\overline{\Pi}_{\rho}\|_{L^{5/4,15/7}(Q^{+}(\rho))}
\lesssim \rho^{-1}\|\nabla\Pi \|_{L^{5/4}(Q^{+}(\rho))}.
\]
Combining the two estimates yields the result.
\end{proof}

\section{Proof of Theorem \ref{the1.1}}
\label{S5}

In this section, we consider the upper box-dimension of potential interior singular points for suitable weak solutions to the Navier--Stokes equations \eqref{NS}.
Actually, Theorem \ref{the1.1} is a direct consequence of Proposition \ref{boxprop}.
The proof is based the contradiction argument, which is very well-known (see, for example, \cite{[WW2],[KP],[CY18]}).
Thus, we suffices to prove Proposition \ref{boxprop}.
We divide its proof into a few steps.

\begin{enumerate}[\bf{Step} 1)]
\item
We may prove the theorem with assuming that $z=0$.
We shall show that the quantities on the left of \eqref{loc1.222} can be controlled by the following assumption \eqref{assumei}, which was used in \cite[Page, 1762, inequality (3.2)]{[WW2]}.
Assume that for some fixed $\rho \le \rho_0$
\begin{equation}
\label{assumei}
\int_{Q (2\rho)} |\nabla u |^{2} + | u |^{ 10/3} + |\Pi-\overline{\Pi}_{2\rho} |^{ 5/3} + |\nabla \Pi| ^{5/4} dxdt
\leq (2\rho)^{ 5/3-\gamma}\ep_{1}
\end{equation}
From the local energy inequality we obtain that
\begin{equation}
\label{E51}
\begin{split}
&\sup_{-\rho^{2}\leq t<0}\int_{B(\rho )}|u|^{2} dx
+ \int_{Q (\rho)} |\nabla u|^{2}dxdt \\
&\lesssim \left(\int_{Q(2\rho)}|u|^{10/3}dxdt\right)^{3/5}
+ \int_{Q(2\rho)} |u|^{2} u \cdot \nabla \phi dxdt \\
&\quad + \int_{Q(2\rho)}(\Pi-\overline{\Pi}_{2\rho}) u \cdot \nabla \phi dxdt
\end{split}
\end{equation}
by fixing $\phi(x,t)$, which is a smooth positive function supported in $Q(2\rho)$ and with value $1$ on the ball $Q(\rho)$.
\item
Using the divergence free condition, H\"older's inequality, and the Gagliardo--Nirenberg inequality, we estimate
\begin{equation}
\label{wwnon1}
\begin{split}
&\int_{Q(2\rho)} |u|^{2} u \cdot \nabla \phi dxdt \\
&\lesssim \rho^{-1}\Big\| |u|^{2}- \overline{|u|^{2}}_{2\rho}  \Big\|_{L^{10/7,15/8}(Q(2\rho))} \norm{u}_{L^{10/3,15/7}(Q(2\rho))}\\
&\lesssim \rho^{-1}\| u ^{2}\|_{L^{5/3}(Q(2\rho))}^{1/2} \|u\nabla u  \|^{1/2}_{L^{5/4}(Q(2\rho))}  \norm{u}_{L^{10/3,15/7}(Q(2\rho))}\\
&\lesssim \rho^{-1/2}\| u  \|^{5/2}_{L^{10/3}(Q(2\rho))} \norm{\nabla u}^{1/2}_{L^{2}(Q(2\rho))},
\end{split}
\end{equation}
and
\begin{equation}
\label{wwnon2}
\begin{split}
&\int_{Q(2\rho)}(\Pi-\overline{\Pi}_{2\rho}) u \cdot \nabla \phi dxdt \\
&\lesssim \rho^{-1}\norm{\Pi-\overline{\Pi}_{2\rho}}_{L^{10/7,15/8}(Q(2\rho))} \norm{u}_{L^{10/3,15/7}(Q(2\rho))} \\
&\lesssim \rho^{-1}\|\Pi-\overline{\Pi}_{2\rho}\|_{L^{5/3}(Q(2\rho))}^{1/2}
\|\Pi-\overline{\Pi}_{2\rho}\|^{1/2}_{L^{5/4,15/7}(Q(2\rho))} \norm{u}_{L^{10/3,15/7}(Q(2\rho))} \\
&\lesssim \rho^{-1/2}\|\Pi-\overline{\Pi}_{2\rho}\|_{L^{5/3}}^{1/2}
\|\nabla\Pi\|^{1/2}_{L^{5/4}(Q(2\rho))} \norm{u}_{L^{10/3}(Q(2\rho))}.
\end{split}
\end{equation}
Combining \eqref{E51}, \eqref{wwnon1}, and \eqref{wwnon2} and using the assumption \eqref{assumei}, we get
\begin{equation}
\label{loc1.222}
\begin{split}
&\sup_{-\rho^{2}\leq t<0}\int_{B(\rho )}|u|^{2} dx
+ \int_{Q (\rho)} |\nabla u|^{2}dxdt \\
&\lesssim \norm{u}^{2}_{L^{10/3}(Q(2\rho))}
+ \rho^{-1/2} \|u\|^{5/2}_{L^{10/3}(Q(2\rho))} \norm{\nabla u}^{1/2}_{L^{2}(Q(2\rho)) } \\
&\quad + \rho^{-1/2}\| u  \|_{{L^{10/3}}(Q(2\rho))} \|\Pi-\overline{\Pi}_{2\rho}\|_{L^{5/3}(Q(2\rho))}^{1/2} \|\nabla\Pi\|^{1/2}_{L^{5/4}(Q(2\rho))} \\
&\lesssim \ep_{1}^{3/5} \rho^{1-3\gamma/5} + \ep_{1} \rho^{7/6- \gamma} \\
&\lesssim \ep_{1}^{3/5} \rho^{7/6- \gamma},
\end{split}
\end{equation}
where we have used $\gamma \geq 5/12$.
Hence we get
\begin{equation}
\label{E52}
A(u,\rho) \lesssim \ep_{1}^{3/5}\rho^{1/6-\gamma}.
\end{equation}
\item
If we take
\[
\beta = 1/6 \qand \theta=\rho^{\beta}< 1/8,
\]
then from \eqref{cp}, \eqref{assumei}, and \eqref{E52} we get
\begin{align*}
E_{2}(\theta\rho)
&\lesssim \theta^{-1} E(\rho) + \theta^{2} A(\rho) \\
&\lesssim \ep_{1} \theta^{-1} \rho^{2/3-\gamma} + \ep_{1}^{3/5}\theta^{2} \rho^{1/6-\gamma } \\
&\lesssim \ep_{1}^{3/5} \rho  ^{-\beta + 2/3 - \gamma} + \ep_{1}^{3/5} \rho^{2\beta + 1/6 - \gamma} \\
&\lesssim \ep_{1}^{3/5} \rho^{1/2 - \gamma}.
\end{align*}
If $\gamma \leq 1/2$, then
\begin{equation}
\label{condtion1}
E_{2}(\theta\rho) \lesssim \ep_{1}^{3/5}.
\end{equation}
\item
From Lemma \ref{presurel} and \eqref{assumei} we have
\begin{equation}
\label{condtion3}
\begin{split}
&E(\theta\rho) + P_{1,3/2}(\nabla\Pi,\theta\rho) \\
&\lesssim \theta^{-1} E (\rho) + \theta P_{5/4}(\nabla\Pi,\rho) \\
&\lesssim \ep_{1}\theta^{-1}\rho^{2/3-\gamma}
+ \ep_{1}^{4/5} \theta \rho^{1/3-4\gamma/5} \\
&\lesssim \ep_{1}^{4/5} \rho^{-1/6+2/3-\gamma}
+ \ep_{1}^{4/5} \rho^{1/6+1/3-4\gamma/5} \\
&\lesssim \ep_{1}^{4/5} \rho^{1/2-\gamma}.
\end{split}
\end{equation}
Combining \eqref{condtion1} and \eqref{condtion3}, we get
\[
E(\theta\rho)+E_{2}(\theta\rho)
+P_{5/4}(\nabla\Pi,\theta\rho)
\lesssim \ep_{1}^{3/5}.
\]
This together with \eqref{special} yields $z=0$ is a regular point.
This completes the proof of Proposition \ref{boxprop}.
\end{enumerate}
\qed

\section{Proof of Theorem \ref{the1.2}}
\label{S6}

In the first place, we prove Proposition \ref{lebp}.
Then Theorem \ref{brct} and Theorem \ref{brc} follow from Proposition \ref{lebp} and \eqref{gktb}.
After that, we shall give the proof of Proposition \ref{boxpropb}.
Then Theorem \ref{the1.2} is a direct consequence of Proposition \ref{boxpropb} by the contradiction argument.

\begin{prop}
\label{lebp}
Let $p,q$ be defined in Theorem \ref{brc} and denote $\alpha=\frac{2}{\frac{2}{p}+\frac{3}{q}}$.
Suppose that $(u,\Pi)$ is a suitable weak solution to the Navier--Stokes equations \eqref{NSB} in $Q^{+}(R)$.
Then there holds, for any $R>0$,
\begin{equation}
\label{key ineq}
\begin{split}
&\|u\|^2_{L^{\infty,2}(Q^{+}(R/2))} + \|\nabla u\|^2_{L^{2}(Q^{+}(R/2))} \\
&\lesssim R^{(3\alpha-4)/\alpha} \|u\|^{2}_{L^{p,q}(Q^{+}(R))}
+ R^{(3\alpha-5)/(\alpha-1)} \|u\|^{2\alpha/(\alpha-1)}_{L^{p,q}(Q^{+}(R))} \\
&\quad + R^{-1} \|\Pi-\overline{\Pi}_{B^{+}(R)}\|^{2}_{L^{\lambda,\kappa^*}(Q^{+}(R))}
\end{split}
\end{equation}
where the implied positive constants does not depend on $R$.
\end{prop}

We divide its proof into a few steps.

\begin{proof}
\begin{enumerate}[\bf{Step} 1)]
\item
Fix $r$ and $\rho$ satisfying
\[
0<R/2\leq r<\rho\leq R.
\]
Let $\phi(x,t)$ be non-negative smooth function supported in $Q(\rho)$ such that
$\phi(x,t) = 1$ on $Q(r)$ and
\[
|\partial_{t}\phi| + |\nabla \phi|^2 + |\nabla^{2}\phi| \lesssim (\rho-r)^{-2}.
\]
From the local energy inequality \eqref{loc} we have
\begin{equation}
\label{loc2}
\int_{B^{+}(\rho)} |u(x,t)|^{2} \phi(x,t) dx + 2\int_{Q^{+}(\rho)} |\nabla u|^{2}\phi  dxds
\lesssim  L_{1} +L_{2}+L_{3}
\end{equation}
where
\begin{align*}
L_{1} &= \frac{1}{(\rho-r)^2} \int_{Q^{+}(\rho)} |u|^{2} dxds, \\
L_{2} &= \frac{1}{(\rho-r)} \int_{Q^{+}(\rho)} |u|^{3} \phi dxds, \\
L_{3} &= \frac{1}{(\rho-r)} \int_{Q^{+}(\rho)} u(\Pi-\overline{\Pi}_{B^{+}(\rho)} ) dxds.
\end{align*}
In order to estimate $L_{1},L_{2}$, and $L_{3}$, we shall use the following interpolation inequality.
If $k$ and $l$ satisfy
\[
\frac{2}{l} + \frac{3}{k} = \frac{3}{2} \qand 2\leq l\leq\infty,
\]
then
\begin{equation}
\label{sampleinterplation}
\begin{split}
\norm{u}_{L^{k,l}(Q(\mu))}
&\lesssim \norm{u}_{L^{\infty,2} (Q(\mu))}^{1-2/l}\norm{u}_{L^{ 2,6 }(Q(\mu))}^{2/l} \\
&\lesssim \norm{u}_{L^{\infty,2} (Q(\mu))}^{1-2/l}(\norm{u}_{L^{\infty,2} (Q(\mu))} + \norm{\nabla u}_{L^2(Q(\mu))})^{2/l} \\
&\lesssim \norm{u}_{L^{\infty,2} (Q(\mu))} +\norm{\nabla u}_{L^2(Q(\mu))},
\end{split}
\end{equation}
\item
By H\"older's inequality and \eqref{sampleinterplation}, we have
\begin{align*}
L_{1}
&\lesssim \frac{\rho^{5/3}}{(\rho-r)^{2}}\left(\int_{Q^{+}(\rho)} |u|^{3} dxds\right)^{2/3} \\
&\lesssim \frac{\rho^{5/3}}{(\rho-r)^{2}} \rho^{(\alpha-1)} \norm{u}_{L^{p,q}(Q^{+}(\rho))}^{2\alpha/3}
\left(\norm{u}_{L^{\infty,2}(Q^{+}(\rho))}^{2}+\norm{\nabla u}_{L^{2}(Q^{+}(\rho))}^{2}\right)^{(3-\alpha)/3}.
\end{align*}
By Young's inequality, there is a positive constant $C$ such that
\begin{equation}
\label{4.4.6}
L_{1}
\le\frac{1}{16}\left(\norm{u}_{L^{\infty,2}(Q^{+}(\rho))}^{2}+\norm{\nabla u}_{L^{2}(Q^{+}(\rho))}^{2}\right)
+ \frac{C\rho^{5/\alpha}}{(\rho-r)^{6/\alpha}} \rho^{3(\alpha-1)/\alpha} \norm{u}_{L^{p,q}(Q^{+}(\rho))}^{2}.
\end{equation}
Similarly, by H\"older's inequality and \eqref{sampleinterplation}, we have
\begin{align*}
L_{2}
&= \frac{1}{(\rho-r)} \int_{Q^{+}(\rho)}|u|^{\alpha}|u|^{3-\alpha}dxdt \\
&\lesssim \frac{1}{(\rho-r)} \norm{u}_{L^{p,q}(Q^{+}(\rho))}^{\alpha}\|u\|^{3-\alpha}
_{L^{(3-\alpha)(\frac{p}{\alpha})^{\ast},(3-\alpha)(\frac{q}{\alpha})^{\ast}}(Q^{+}(\rho))} \\
&\lesssim \frac{\rho^{3(\alpha-1)/2}}{(\rho-r)}\norm{u}_{L^{p,q}(Q^{+}(\rho))}^{\alpha}\|u\|^{3-\alpha}
_{L^{2(\frac{p}{\alpha})^{\ast},2(\frac{q}{\alpha})^{\ast}}(Q^{+}(\rho))} \\
&\lesssim \frac{\rho^{3(\alpha-1)/2}}{(\rho-r)} \norm{u}_{L^{p,q}(Q^{+}(\rho))}^{\alpha}
\left(\norm{u}_{L^{\infty,2}(Q^{+}(\rho))}^{2}+\norm{\nabla u}_{L^{2}(Q^{+}(\rho))}^{2}\right)^{(3-\alpha)/2}.
\end{align*}
By Young's inequality, there is a positive constant $C$ such that
\begin{equation}
\label{wwz}
L_{2}
\le \frac{1}{16} \left(\norm{u}_{L^{\infty,2}(Q^{+}(\rho))}^{2}+\norm{\nabla u}_{L^{2}(Q^{+}(\rho))}^{2}\right)
+ \frac{C\rho^{3 }}{(\rho-r)^{\frac{2}{\alpha-1}}} \norm{u}_{L^{p,q}(Q^{+}(\rho))}^{2\alpha/(\alpha-1)}.
\end{equation}
Finally, by H\"older's inequality and \eqref{sampleinterplation}, we have
\begin{align*}
L_{3}
&\lesssim \frac{1}{(\rho-r)}\|\Pi-\overline{\Pi}_{B^{+}(\rho)}\|_{L^{\lambda,\kappa^*}(Q^{+}(\rho))} \norm{u}_{L^{i,j}(Q^{+}(\rho))} \\
&\lesssim \frac{\rho^{1/2} }{(\rho-r)} \|\Pi-\overline{\Pi}_{B^{+}(\rho)}\|_{L^{\lambda,\kappa^*}(Q^{+}(\rho))}
\left(\norm{u}_{L^{\infty,2}(Q^{+}(\rho))}^{2}+\norm{\nabla u}_{L^{2}(Q^{+}(\rho))}^{2}\right)^{1/2}
\end{align*}
where $\kappa^*$, $\lambda$, $i$, and $j$ are numbers in \eqref{gktc1} and \eqref{reg:20}.
By Young's inequality, there is a positive constant $C$ such that
\begin{equation}
\label{4.4.5}
L_{3}
\le \frac{1}{16} \left(\norm{u}_{L^{\infty,2}(Q^{+}(\rho))}^{2}+\norm{\nabla u}_{L^{2}(Q^{+}(\rho))}^{2}\right)
+ \frac{C\rho}{(\rho-r)^{2}} \|\Pi-\overline{\Pi}_{B^{+}(\rho)}\|^{2}_{L^{\lambda,\kappa^*}(Q^{+}(\rho))}.
\end{equation}
\item
Comgining the inequalities \eqref{loc2}, \eqref{wwz}, \eqref{4.4.5}, and \eqref{4.4.6}, we obtain that
\begin{align*}
&\sup_{-r^{2}\leq t\leq0}\int_{B(r)}|u|^2 dx + \int_{Q^{+}(r)}\big|\nabla u\big|^2  d  x d  \tau \\
&\le \frac{1}{4} \left(\norm{u}_{L^{\infty,2}(Q^{+}(\rho))}^{2}+\norm{\nabla u}_{L^{2}(Q^{+}(\rho))}^{2}\right)
+ \frac{C\rho^{5/\alpha}}{(\rho-r)^{6/\alpha}} \rho^{3(\alpha-1)/\alpha} \norm{u}_{L^{p,q}(Q^{+}(\rho))}^{2} \\
&\quad +\frac{C\rho^3}{(\rho-r)^{2/(\alpha-1)}} \norm{u}_{L^{p,q}(Q^{+}(\rho))}^{2\alpha/(\alpha-1)}
+\frac{C\rho}{(\rho-r)^{2}} \|\Pi-\Pi_{B(\rho)}\|^{2}_{L^{\lambda,\kappa^*}(Q^{+}(\rho))}.
\end{align*}
Applying the standard iteration argument, Lemma V.3.1 in \cite{[Giaquinta]}, we get the result.
\end{enumerate}
\end{proof}

We end this section by giving the proof of Proposition \ref{boxpropb}.

\begin{proof}[Proof of Proposition \ref{boxpropb}]
We may assume $z=0$.
First we assume that
\begin{equation}
\label{assume}
\int_{Q^{+} (2\rho)}
|\nabla u |^{2} +| u |^{ 10/3}+|\Pi-\overline{\Pi}_{2\rho} |^{ 5/3}
+ |\nabla \Pi| ^{5/4}dxdt \leq    (2\rho)^{ 5/3-\gamma}\ep_{2}.
\end{equation}
Following the argument in \cite{[KY16]} we shall determine a suitable parameter $\gamma$.
From \eqref{assume}
\begin{equation}
\label{tip54}
\rho^{-1}\|\nabla\Pi\|_{L^{5/4}(Q^{+}(\rho))} \lesssim \ep_{2}^{4/5} \rho^{1/3 - 4\gamma/5}.
\end{equation}
By the Poincar\'e inequality
\[
\int_{Q^{+} (2\rho)}|u|^{2}(\partial_t \phi+\Delta\phi)
\lesssim \int_{Q^{+} (2\rho)} |\nabla u|^{2}
\lesssim \ep_{2} \rho^{5/3-\gamma}.
\]
This and \eqref{wwnon1} and \eqref{wwnon2} implies that
\begin{equation}
\label{loc1.2}
\sup_{-\rho^{2}\leq t<0} \int_{B^{+}(\rho )}|u|^{2} dx + \int_{Q^{+} (\rho)} |\nabla u|^{2} dxdt
\lesssim \rho^{5/3-\gamma} + \rho^{7/6-\gamma}
\lesssim \rho^{7/6-\gamma}.
\end{equation}
Consequently,
\begin{equation}
\label{E61}
A(u,\rho) \lesssim \ep_{2}\rho^{1/6-\gamma}.
\end{equation}
If we take
\[
\beta=\frac{1}{4\lambda}-\frac{1}{12}-\frac{\gamma}{10} \qand \theta=\rho^{\beta} < \frac{1}{4},
\]
then from \eqref{C3}, \eqref{pbd}, and \eqref{tip54} we conclude that
\begin{equation}
\label{E62}
\begin{split}
&E(\theta\rho)+\tilde{D}_1(\theta\rho) \\
&\lesssim \theta^{-1} E(\rho) + \theta^{-1}  E^{1/\lambda}(\rho) A^{1-1/\lambda}(\rho)
+ \theta E^{1/2}(\rho) + \theta P_{5/4}(\nabla\Pi,\rho) \\
&\lesssim \ep_{2}\rho^{-\beta+2/3-\gamma}
+ \ep_{2} \rho^{-\beta+(1-1/\lambda)(1/6-\gamma) + (2/3-\gamma)/\lambda}
+ \ep_{2}^{1/2}\rho^{\beta+(2/3-\gamma)/2}
+ \ep_{2}^{4/5}\rho^{\beta+1/3- 4\gamma/5} \\
&\lesssim \ep_{2}\rho^{-\beta+2/3-\gamma}
+ \ep_{2}^{1/2}\rho^{\beta+(2/3-\gamma)/2}
+ \ep_{2}^{4/5}\rho^{\beta+1/3- 4\gamma/5} \\
\end{split}
\end{equation}
since our choice of $\beta$ satisfies
\[
-\beta+(1-1/\lambda)(1/6-\gamma) + (2/3-\gamma)/\lambda = \beta+1/3- 4\gamma/5.
\]
All the exponents of $\rho$ on the right of \eqref{E62} are nonnegative if
\[
\gamma \le \min\set{\frac{5(3\lambda-1)}{18\lambda}, \frac{5(1+\lambda)}{12\lambda}, \frac{5(1+\lambda)}{18\lambda}}.
\]
Thus, we can choose $\lambda$ sufficiently close to $1$ to obtain $\gamma \leq 5/9$ and
\[
E (\theta\rho)+\tilde{D}_1(\theta\rho) \lesssim \ep_{2}^{1/2}.
\]
Hence Theorem \ref{brct} implies that $z=0$ is a regular point.
\end{proof}

\section*{Acknowledgement}

The research of Wang was partially supported by the National Natural Science Foundation of China under grant No. 11601492 and the
the Youth Core Teachers Foundation of Zhengzhou University of Light Industry.
Yang has been supported by the National Research Foundation of Korea (NRF) grant funded by the Korea government(MSIP) (No. 2016R1C1B2015731) and (No. 2015R1A5A1009350).

\end{document}